\documentclass{article}

\usepackage{latexsym,enumerate}
\usepackage{amsmath,amsthm,amsopn,amstext,amscd,amsfonts,amssymb}
\usepackage[ansinew]{inputenc}
\usepackage{verbatim}
\usepackage{pstricks}
\usepackage{wasysym}
\usepackage[all]{xy}
\usepackage{epsfig} 
\usepackage{graphicx,psfrag} 
\usepackage{subfigure}
\usepackage{authblk}

\newtheorem{theorem}{Theorem}[section]

\newtheorem{proposition}[theorem]{Proposition}
\newtheorem{corollary}[theorem]{Corollary}

\newtheorem{example}[theorem]{Example}
\newtheorem{remark}[theorem]{Remark}

\newcommand{\spb}[1]{\smallskip}
\newcommand{\mpb}[1]{\medskip}
\newcommand{\bpb}[1]{\bigskip}

\renewcommand{\a}{\alpha}

\renewcommand{\d}{\delta}
\newcommand{\D}{\Delta}

\begin{document}

\title{Computing the strong alliance polynomial of a graph}

\author[1]{Walter Carballosa\thanks{waltercarb@gmail.com}}
\affil[1]{Facultad de Matem\'aticas, Universidad Aut\'onoma de Guerrero, Carlos E. Adame No.54 Col. Garita, 39650 Acalpulco Gro., Mexico.}

\author[1]{Juan Carlos Hern\'andez-G\'omez\thanks{carloshg@cimat.mx}}

\author[2]{Omar Rosario\thanks{orosario@math.uc3m.es}}
\affil[2]{Departamento de Matem\'aticas, Universidad Carlos III de Madrid, Avenida de la Universidad 30, 28911 Legan\'es, Madrid, Spain.}

\author[2]{Yadira Torres-Nu\~nez\thanks{yadiratn83@gmail.com}}

\maketitle{}

\begin{abstract}
We introduce the strong alliance polynomial of a graph. The strong alliance polynomial of a graph $G$ with order $n$ and strong defensive alliance number $a(G)$ is the polynomial $a(G;x):=\sum_{i=a(G)}^{n}\, a_i(G)\ x^i$, where $a_{k}(G)$ is the number of strong defensive alliances with cardinality $k$ in $G$.
We obtain some properties of $a(G; x)$ and its coefficients.
In particular, we compute strong alliance polynomial for path, cycle, complete, start, complete bipartite and double star graphs; some of them verify unimodality.


\end{abstract}
%
%

\section{Introduction.}

The study of the mathematical properties of alliances in graphs were started by P. Kristiansen, S. M. Hedetniemi and  S. T. Hedetniemi in 2004, see \cite{KHH}.
The alliances in graphs is a topic of recent and increasing interest in
graph theory; see, for instance \cite{C,FLH3,H3,RVS1,RVGS,S,SBF1,SRV1,SRV2}.

Some parameters of a graph $G$ allow to define polynomials on the graph $G$, for instance,
the parameters associated to matching sets \cite{F,GG}, independent sets \cite{BDN,HL}, domination sets \cite{AAP,AB}, chromatic numbers \cite{R,T} and many others.
In \cite{CRST}, the authors use the exact index of alliance by define the exact alliance polynomial of a graph. In this work we choose the cardinality of the strong alliance by define the strong alliance polynomial of a graph (see Section 2).

We begin stablishing the used terminology.  Throughout
this paper, $G=(V,E)$ denotes a 
simple graph of order $|V|=n$ and size $|E|=m$.
We denote two adjacent vertices $u$ and $v$ by
$u\sim v$. For a nonempty set $X\subseteq V$, and a vertex $v\in V$,
 $N_X(v)$ denotes the set of neighbors that $v$ has in $X$:
$N_X(v):=\{u\in X: u\sim v\},$ and the degree of $v$ in $ X$ will be
denoted by $\delta_{X}(v)=|N_{X}(v)|$.
We denote the degree of a vertex $v_i\in V$  by $\delta(v_i)=\delta_{G}(v_i)$ (or by $\delta_i$ for short),
the minimum degree of $G$ by $\d$ and the maximum degree of $G$ by $\D$.
The subgraph induced by
$S\subset V$ will be denoted by  $\langle S\rangle $ and the
complement of the set $S\subset V$ will be denoted by $\overline{S}$.

A nonempty set $S\subseteq V$ is a \emph{defensive  $k$-alliance} in
$G$, with $k\in [-\Delta,\Delta]\cap \mathbb{Z}$, if for
every $ v\in S$,
\begin{equation}\label{cond-A-Defensiva} \delta _S(v)\ge \delta_{\overline{S}}(v)+k.\end{equation}


The concept of defensive $k$-alliance was introduced in \cite{SD} as a generalization of defensive alliance defined in \cite{KHH}.
Colloquially speaking, a defensive $k$-alliance in a graph $G$ is a set $S$ of vertices of $G$ such that every vertex in $S$ has at least $k$ more neighbors in $S$ than it has outside of $S$.
A \emph{defensive alliance} in a graph $G$ is a defensive ($-1$)-alliance in $G$, while a \emph{strong defensive alliance} is a defensive $0$-alliance.
In strong defensive alliance $S$ says that each vertex of $S$ is strongly held from a possible attack by the vertices in $\overline{S}$ to be in numerical superiority over its neighbors in $\overline{S}$.
The strong defensive alliances are also known as cohesive set, see e.g., \cite{SD}.
For graphs having strong defensive $k$-alliances, the \emph{defensive $k$-alliance
number} of $G$, denoted by $\a_k(G)$, is defined as the minimum cardinality of a defensive $k$-alliance in $G$. In particular, we denote by $a(G)$ the \emph{strong defensive alliance
number} (i.e., defensive $0$-alliance number) of a graph $G$.

Throughout this paper we just consider defensive $k$-alliances with $k = 0$, i.e., strong defensive alliances.
Also, we pay special attention to strong defensive alliances with connected induced subgraph.
Note that to study $S\subset V$ with $\langle S\rangle$ no connected, can be analyzed separately in each connected components.
Throughout this paper we consider strong defensive alliance $S$ with connected induced subgraph.

\smallskip

  A finite sequence of real numbers $(a_{0} , a_{1} , a_{2} , . . . , a_{n})$ is said to be \emph{unimodal} if there is some $k \in \{0, 1, . . . , n\}$, called the \emph{mode} of the sequence, such that
  $$a_{0} \leq . . . \leq a_{k-1} \leq a_{k} \quad \text{and} \quad a_{k} \geq a_{k+1} \geq . . . \geq a_{n};$$
  the mode is unique if $a_{k-1} < a_{k}$ and $a_{k} > a_{k+1}$.
  A polynomial is called unimodal if the sequence of its coefficients is unimodal.

In the next section, we introduce the strong alliance polynomial and obtain some of its properties.
In Section 3, we compute the strong alliance polynomial for some graphs and study its coefficients; in particular, we show that some of them are unimodal.

\
\section{Strong alliance polynomial}

In this section, we state the definition of strong alliance polynomial and some
of its properties.

Let $G$ be a graph with order $n$ and strong defensive alliance number $a(G)$.
We define the \emph{strong alliance polynomial} of $G$ with variable $x$ as follows:

\begin{equation}\label{eq:Poly2}
    a(G;x)= \displaystyle\sum_{i=a(G)}^{n} a_k(G) x^{i},
\end{equation}
where $a_{k}(G)$ is the number of strong defensive alliances with cardinality $k$ in $G$.

\smallskip

\begin{remark}\label{r:exit}
  For every graph $G$, there is $S\subset V$ with $\langle S\rangle$ is a connected component of $G$. Thus, $S$ in a strong defensive alliance in $G$ and $a(G;x)\neq0$.
\end{remark}

The cycle graph $C_4$ with $4$ vertices, for example, has one strong defensive alliance of cardinality $4$, four strong defensive alliances of cardinalities $3$ and $2$; its strong alliance polynomial is then $a(C_4;x) = x^4+ 4x^3+ 4x^2$. As another example, it is easy to see that the path graph $P_4$ with $4$ vertices has strong alliance polynomial $a(P_4;x)= x^4 + 2 x^3 + 3 x^2$.

\smallskip

An \emph{isomorphism of graphs} $G_1$ and $G_2$ is a bijection between the vertex sets of $G_1$ and $G_2$, $f: V(G_1)\Longrightarrow V(G_2)$ such that any two vertices $u$ and $v$ of $G_1$ are adjacent in $G_1$ if and only if $ƒ(u)$ and $ƒ(v)$ are adjacent in $G_2$. If an isomorphism exists between $G_1$ and $G_2$, then the graphs are called \emph{isomorphic} and we write $G_1 \simeq G_2$.

\begin{remark}
  Let $G_1$ and $G_2$ be isomorphic graphs. Then $a(G_1; x) = a(G_2; x)$.
\end{remark}

\smallskip

The \emph{disjoint union} of graphs, sometimes referred simply as \emph{graph union} is defined as follows. For two graphs $G_1=(V_1,E_1)$ and $G_2=(V_2,E_2)$ with disjoint vertex sets $V_1$ and $V_2$ (and hence disjoint edge sets), their union is the graph $G_1\cup G_2:=(V_1\cup V_2, E_1 \cup E_2)$. It is a commutative and associative operation.

\begin{remark}\label{r:disjUnion}
 If $G$ is not a connected graph, the problem of determine $a(G;x)$ is reduced to determine the strong alliance polynomials of each connected component. Thus, $a(G;x)$ is the sum of the polynomials of each connected component. In other word, if $G=G_1\cup G_2\cup\ldots\cup G_r$, then
 \[
 a(G;x)=\displaystyle\sum_{i=1}^{r} \ a(G_i;x).
 \]
\end{remark}

The $n$-vertex edgeless graph or \emph{empty graph} is the complement graph for the complete graph $K_n$, and therefore it is commonly denoted as $E_n$ for $n\ge1$.
Note that, for $n\ge1$, we have $a(E_n; x) = n x$.

\begin{corollary}\label{c:vacio}
 Let $n$ be a natural number with $n\ge1$.  If $a(G; x) = n x$, then $G\simeq E_n$.
\end{corollary}

\begin{proof}
Assume first that $G$ is not isomorphic graph to $E_n$. Since term $nx$ appear in $a(G;x)$, the graph $G$ has $n$ isolate vertices. So, $a(G;x)=a(E_n;x) + a(G^*;x)$ where $G^*$ is a graph such that $G=E_n\cup G^*$. Thus, $a(G;x) \neq a(E_n;x)$ since Remark \ref{r:exit} gives $a(G^*;x)\neq 0$.
\end{proof}

\medskip

The following proposition shows general properties which satisfy the strong alliance polynomials.

\begin{proposition}\label{p:AlliPoly}
   Let $G$ be a graph. Then, $a(G;x)$ satisfies the following properties:

   \begin{enumerate}[i)]
     \item {All real zeros of $a(G;x)$ are non-positive numbers.}

     \item {The value $0$ is a zero of $a(G;x)$ with multiplicity $a(G)>0$.}

     \item {$a(G;1) < 2^{n}$, and it is the number of strong defensive alliance in $G$.}
   \end{enumerate}
\end{proposition}

\begin{proof}
Since the coefficients of $a(G;x)$ are non-negatives, we have (1). We have a common factor $x^{a(G)}$ in $a(G;x)$ and $a(G)>0$.
By \eqref{eq:Poly2}, we have $a(G;1)= \displaystyle\sum_{i=a(G)}^{n}  a_i(G)$.
Thus, $a(G;1)$ has as upper bound the number of connected induced subgraph in $G$; this amount is less that $2^{n}$, since we have $2^{n}-1$ nonempty subsets of $V$.
\end{proof}

\medskip

The following theorem gives some properties on coefficients of the strong alliance polynomial of a graph.

\begin{theorem}\label{th:prop}
Let $G$ be a graph. Then, the following properties are satisfied

\begin{enumerate}[i)]
\item $0 \leq a_2(G) \leq m$; and $a_2(G) =m $ if and only if $G$ is a cycle graph, a path graph, or disjoint union of some of them.
\item $a_k(G) = 0$ if and only if there is $v\in S$ such that $\d_S(v) \leq \left \lfloor \frac{\d (v) -1}{2} \right \rfloor$ in every $S \subset V(G)$ with $|S|=k$.
\item $a_n(G) = 1$ if and only if $G$ is connected.
\item $a_2(G) = 1$ if and only if there exists an unique edge $uv \in E(G)$ with $ \d (u), \d (v)\leq 2$.
\end{enumerate}
\end{theorem}

\begin{proof}
We prove separately each item.
\begin{enumerate}[$i)$]
\item  On the one hand, it is clear that $a_k(G)\geq 0$, in particular $a_2(G) \geq 0$. Besides, since $a_2(G)$ is the number of subsets of two vertices which are strong defensive alliance in $G$, then the amount of these subsets has as upper bound the size of $G$, and so, $a_2(G)\leq m$.

    On the other hand, notice that if $G$ is no connected then the proof can be reduced to analyze each connected component. Without loss of generality we can assume that $G$ is connected.
    First, we claim that $a_2(G) =m$ if and only if $\delta (v_i) \leq 2$ for every  $v_i \in V$.
    Let $uv\in E$ with $S:=\{u,v\}$ a strong defensive alliance in $G$; hence, we have $\d_S(u)=1=\d_S(v)$, $1=\delta_S(u) \geq \delta_{\overline{S}}(u)$ and $1=\delta_S(v) \geq \delta_{\overline{S}}(v)$, thus, we have $\delta(u) \leq 2$ and $\delta(v) \leq 2$.
    If $\delta (v_i) \leq 2$ for every $v_i \in V$, then taking $S :=\{ u,v \}\subset V$ with  $u\sim v$ we have that $\delta_S(u)+ \delta_{\overline{S}}(u) = \delta(u) \leq 2$, and so, $\d_{\overline{S}}(u) \leq 1 = \delta_S(u)$; analogously we obtain $\d_{\overline{S}}(v) \leq \d_S(v)$; in fact, $S$ is a strong defensive alliance in $G$. Thus, we have proved the claim.
    Note that it suffices prove that $G$ is an isomorphic graph to a path or cycle graph if and only if $\d(v_i)\leq 2$ for every $v_i \in V$.
    It is clear that, if $G\simeq P_n$ or $G\simeq C_n$ for any $n\in\mathbb{N}$ then $a_2(G)=m$.
    Assume now that $\d(v_i)\leq 2$ for every $v_i \in V$.
    If $\d(v_1)=1$, then there is $v_2\in V$ with $v_1\sim v_2$. Hence, if $\d(v_2)=1$ then $G\simeq P_2$, else $\d(v_2)=2$ and there is $v_3$ with $v_2 \sim v_3$. Now, if $\d(v_3)=1$ then $G\simeq P_3$, else $\d(v_3)=2$ and iterating this process we obtain a vertex $v_n$ with $\d(v_n)=1$, and so, $G\simeq P_n$. Similarly, we may obtain the result when there no is $v\in V$ with $\d(v)=1$, i.e., $G\simeq C_n$.

\item On the one hand, if $a_k(G) = 0$ for any $k\in \mathbb{N}$, then there is no $S\subset V$ with $|S|=k$ which is strong defensive alliance in $G$. Hence, for every $S\subset V$ with $|S|=k$ there is $v\in S$ such that $\d_S(v)< \d_{\overline{S}}(v)$, so, $2\d_S(v) < \d(v)$ and $2\d_S(v) \leq \d(v) - 1$. Thus, since $\d_S(v)$ in an integer number we obtain $\d_S(v)\leq \left \lfloor \frac{\d(v) -1}{2} \right \rfloor$.
    On the other hand, if there is $v\in S$ with $\d_S(v) \leq \left \lfloor \frac{\d (v) -1}{2} \right \rfloor$ in every $S \subset V$ with $|S|=k$, then $\d_S(v) \leq \frac{\d (v) -1}2$ and since $\d_S(u)+\d_{\overline{S}}(u)=\d(u)$ we conclude that $\d_S(v) < \d_{\overline{S}}(v)$, therefore, there is no $S\subset V$ with $|S|=k$ which is strong defensive alliance. Thus, we have $a_k(G)=0$.

\item Note that $V$ is a strong defensive alliance if and only if $G$ is connected, thus, $a_n(G)=1$ if and only if $G$ is connected.


\item  If $a_2(G) = 1$ then there exists an unique edge $uv$ such that $S=\{v,u\}$ is a strong defensive alliance in $G$, besides, using the same argument in item i) we obtain that $\d(u)\le2$ and $\d(v)\leq 2$.
    If $uv\in E$ is the unique edge such that $\d(u), \d(v)\leq 2$, then; $S=\{u,v\}$ is a strong defensive alliance, but, any $S_1=\{u_1, v_1\} \subset V$ with $S_1\neq S$ is no strong defensive alliance since we have either $u_1\nsim v_1$, $\d(u_1)>2$ or $\d(v_1)>2$. So, we have $a_2(G)=1$.
\end{enumerate}
\end{proof}

%
%

\
\section{Strong alliance polynomials of some class of graphs.}

In this Section, we obtain the explicit formula for strong alliance polynomials of some classical class of graphs using combinatorial arguments. Besides, we verify unimodality of strong alliance polynomials of path, cycle, complete and complete bipartite graphs.

\begin{proposition}\label{p:computos}
${}$
\begin{enumerate}
 \item $a(P_n; x) = \sum^{n}_{i=2}(n+1-i)x^{i}$, for $n\ge2$.

 \item $a(C_n; x) = n\sum_{i=2}^{n-1}x^i + x^n$, for $n\ge3$.

 \item $a({K _n; x}) = \sum^{n}_{k=\left\lceil \frac{n+1}{2} \right\rceil} {n \choose k} \ x^{k}$, for $n\ge1$.

\end{enumerate}
\end{proposition}

\begin{proof}
${}$
\begin{enumerate}
\item {We compute the coefficients of $a(P_n; x)$ starting by its strong defensive alliance number $a(P_n)$. Clearly, we have $a(P_n)=2$.
    By Theorem \ref{th:prop} i) we have $a_2(G) = n-1$.
    Assume now that $n\ge3$ and let $k$ be a natural number with $2\le k\le n$.
    Since every $S\subset V(P_n)$ with $|S|=k$ which is strong defensive alliance 
    verify that $\langle S\rangle \simeq P_k$, we have that $a_k(P_n)$ is the number of different $P_k$ in $P_n$. Therefore, we obtain $a_k(P_n)=n-k+1$.}

\item {Similarly to the previous item, we have $a(C_n)=2$. By Theorem \ref{th:prop} iii) we have $a_n(C_n)=1$. Let $k$ be a natural number with $2\le k\le n-1$.
    Since every $S\subset V(C_n)$ with $|S|=k$ which is strong defensive alliance verify that $\langle S\rangle \simeq P_k$, we have that $a_k(C_n)$ is the number of different $P_k$ in $C_n$. Therefore, we obtain $a_k(C_n)=n$.}

\item {Let $S \subset V(K_n)$. So, we have $\d_S(v_i) =|S|-1$ and $\d_S(v_i) + \d_{\overline{S}}(v_i) = n-1$ for every $v_i\in S$. Hence, in order to obtain a strong defensive alliance in $K_n$ it suffices to $\d_S(v_i)\ge \left \lceil \frac{n-1}{2} \right \rceil$ for every $v_i\in S$, and so, we obtain $a(K_n)=\left \lceil \frac{n+1}{2} \right \rceil$.
    Furthermore, if $|S|\ge \left \lceil \frac{n+1}{2} \right \rceil$, then $S$ is a strong defensive alliance in $K_n$; thus, we have $a_k(K_n)={n \choose k}$ for every $\left \lceil \frac{n+1}{2} \right \rceil \le k \le n$.}

\end{enumerate}
\end{proof}

The following results show that many strong alliance polynomials are unimodal.
Note that all of them satisfice that its coefficients are decreasing.

\begin{corollary}\label{c:unimodal}
 The strong alliance polynomials of path, cycle and complete graphs are unimodal.
\end{corollary}


Now, we recall that ${z \choose y}$ is equal zero if $y$ is no integer.

\begin{theorem}\label{t:k_nm}
For $n,m\ge1$, we have
\[
a({K}_{n,m}; x)  = \left( a(K_n;x) +  {n \choose \frac{n}2} x^{\frac{n}2}  \right)
\left( a(K_m;x) +  {m \choose \frac{m}2} x^{\frac{m}2}  \right).
\]
\end{theorem}

\begin{proof}
Denote by $N$ and $M$ the parts of vertices set of the complete bipartite graph $K_{n,m}$, i.e., $N\cup M=V(K_{n,m})$, $|N|=n$ and $|M|=m$. Hence, we have $\d(v)=m$ for every $v\in N$ and $\d(w)=n$ for every $w\in M$.
Clearly, $S\subset V(K_{n,m})$ is a strong defensive alliance in $K_{n,m}$ if and only if there are $S_N\subset N$ and $S_M\subset M$ with $S=S_N\cup S_M$, $|S_N|\ge \lceil \frac{n}{2}\rceil$ and $|S_M|\ge \lceil \frac{m}{2}\rceil$.
Previous statement is straightforward since $\d(v)=|S_M|$ for every $v\in S_N$ and $\d(w)=|S_N|$ for every $w\in S_M$.
Therefore, $a(K_{n,m})=\lceil \frac{n}{2}\rceil + \lceil \frac{m}{2}\rceil$ and for $\lceil \frac{n}{2}\rceil + \lceil \frac{m}{2}\rceil \le k \le n+m$, we have
\[
a_k({K}_{m,n}) = \sum_{i = \lceil \frac{m}{2}\rceil}^{k-\lceil \frac{n}{2}\rceil} {m \choose i} {n \choose k - i}.
\]
Then, we obtain
\[
a({K}_{m,n}; x)  =\sum_{k= \lceil \frac{m}{2}\rceil + \lceil \frac{n}{2}\rceil}^{n+m} \left( \sum _{ i = \lceil \frac{m}{2}\rceil}^{k-\lceil \frac{n}{2}\rceil} {m \choose i} {n \choose k - i}  \right)x^k = \left(\sum_{i=\lceil\frac{n}2\rceil}^{n} {n\choose i} x^i\right) \left(\sum_{j=\lceil\frac{m}2\rceil}^{m} {m\choose j} x^j\right).
\]
\end{proof}

The complete bipartite graph $K_{n-1,1}$ is called an $n$ \emph{star graph} $S_n$, for every $n\ge2$.
We have the following consequence of Theorem \ref{t:k_nm}.

\begin{corollary}\label{c:Sn}
 For $n\ge2$, we have
\[
a(S_n;x)=
\left\{ \begin{array}{ll}
x\, a({K}_{n-1}; x), & \text{if }\  n \text{ is even,} \\
x\, a({K}_{n-1}; x) + {n-1 \choose \frac{n-1}{2}} \ x^{\frac{n+1}{2}}, \quad & \text{if } \ n \text{ is odd.}
\end{array}\right.
\]
\end{corollary}

\begin{remark}\label{r:SnUnim}
 The strong alliance polynomials of star graphs are unimodals. Note that its coefficients are a decreasing subsequence of binomial numbers.
\end{remark}

%

Let us consider two star graphs $S_{r},S_{t}$ with $r,t$ vertices. Denote by $v_r\in V(S_r)$ and $w_t\in V(S_t)$ its central vertices. 
We define the double star graph $S_{r,t}$ as the graph obtained by disjoint union of these star graphs with addition one edge joining $v_r$ and $w_t$, see Figure \ref{fig:S_rt}.

\begin{figure}[h]
\centering
\scalebox{0.7}
{\begin{pspicture}(-1.4,-2.4)(7.4,2.4)
\psline[linewidth=0.04cm,dotsize=0.07055555cm 2.5]{*-*}(2,0)(4,0)
\psline[linewidth=0.04cm,dotsize=0.07055555cm 2.5]{*-*}(0,2)(2,0)(0,0.6)
\psline[linewidth=0.04cm,dotsize=0.07055555cm 2.5]{*-*}(6,2)(4,0)(6,0.6)
\psline[linewidth=0.04cm,dotsize=0.07055555cm 2.5]{*-*}(0,-2)(2,0)(4,0)(6,-2)
\uput[180](0,2){\Large{$v_1$}}
\uput[180](0,0.6){\Large{$v_2$}}
\uput[180](0,-2){\Large{$v_{r-1}$}}
\uput[0](6,2){\Large{$w_1$}}
\uput[0](6,0.6){\Large{$w_2$}}
\uput[0](6,-2){\Large{$w_{s-1}$}}
\uput[270](0,0.2){\Huge{$\cdot$}}
\uput[270](0,-0.2){\Huge{$\cdot$}}
\uput[270](0,-0.6){\Huge{$\cdot$}}
\uput[270](0,-1){\Huge{$\cdot$}}
\uput[270](6,0.2){\Huge{$\cdot$}}
\uput[270](6,-0.2){\Huge{$\cdot$}}
\uput[270](6,-0.6){\Huge{$\cdot$}}
\uput[270](6,-1){\Huge{$\cdot$}}
\end{pspicture}}
\caption{Double star graph $S_{r,t}$ with $r+t$ vertices.} \label{fig:S_rt}
\end{figure}
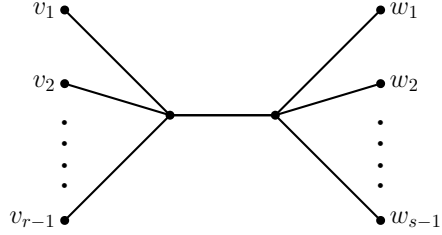

\begin{theorem}\label{t:S_rt}
For $r,t\geq 3$, we have
\[
\begin{aligned}
a(S_{r,t}; x) = \ \ & a(S_r; x) \, + \, a(S_t; x) \,-\,{r-1 \choose (r-1)/2} x^{(r-1)/2}   \,-\,{t-1 \choose (t-1)/2} x^{(t-1)/2} \\
& + \,\left( {r-1 \choose \frac{r}{2}-1} x^{\frac{r}{2}} + a(S_r; x)\right) \left( {t-1 \choose \frac{t}{2}-1} x^{\frac{t}{2}} + a(S_t; x) \right) .
\end{aligned}
\]
\end{theorem}

\begin{proof}
Colloquially speaking, we say that $S_{r,t}$ is formed by two stars $S_r$ and $S_t$ (without to include the edge joining its central vertices).
Denote by $v_r$ and $w_t$ the centers of $S_{r,t}$ with $\d(v_r)=r$ and $\d(w_t)=t$ (the center of $S_r$ and $S_t$, respectively).
Let $S\subset V(S_{r,t})$ be a strong defensive alliance in $S_{r,t}$.
Then $\{v_r,w_t\}\cap S \neq \emptyset$, thus, we have either (1) $v_r\in S$ and $w_t\notin S$, (2) $v_r\notin S$ and $w_t\in S$ or (3) $v_r\in S$ and $w_t\in S$.
Clearly, this cases are disjoints and cases (1) and (2) are symmetric cases.

In order to analyze the firsts cases we can assume that $v_r\in S$ and $w_t\notin S$ (case (1)).
Since $\d(v_r)=r$, we have $\d_S(v_r)\geq \lceil\frac{r}2\rceil$, and so, $|S|\ge \lceil\frac{r+2}2\rceil$. Besides, it is a simple matter to every $S^*\subset V(S_{r,t})$ with $v_r\in S$, $w_t\notin S$ and $|S^*|\ge\lceil\frac{r+2}2\rceil$ is a strong defensive alliance in $S_{r,t}$.
In fact, we obtain in $a(S_{r,t};x)$, from case (1), as addend
$$\displaystyle\sum_{i=\lceil\frac{r}2\rceil}^{r-1} {r-1 \choose i} x^{i+1} = a(S_r;x)-{r-1 \choose (r-1)/2} x^{(r-1)/2}.$$

So, by symmetric we have the analogous result by case (2), i.e., these strong defensive alliances provide in $a(S_{r,t};x)$ as addend $$a(S_t;x)-{t-1 \choose (t-1)/2} x^{(t-1)/2}.$$

In order to finish the proof, we now consider case (3), i.e., $v_r\in S$ and $w_t\in S$.
Obviously, $S\subset V(S_{r,t})$ is a strong defensive alliance in $S_{r,t}$ if and only if there are $S_1\subset V(S_r)$ and $S_2\subset V(S_t)$ with $S=S_1\cup S_2$, $|S_1|\ge \lceil \frac{r}{2}\rceil$ and $|S_2|\ge\lceil\frac{t}{2}\rceil$.
Thus, we obtain in $a(S_{r,t};x)$, from case (3), as addend
$$\left( \sum_{i=\lceil \frac{r}{2}\rceil-1}^{r-1} {r-1 \choose i} x^{i+1} \right) \left( \sum_{j=\lceil \frac{t}{2}\rceil-1}^{t-1} {t-1 \choose j} x^{j+1} \right),$$
or what is the same $\left( {r-1 \choose \frac{r}{2}-1} x^{\frac{r}{2}} + a(S_r; x)\right) \left( {t-1 \choose \frac{t}{2}-1} x^{\frac{t}{2}} + a(S_t; x) \right)$.
\end{proof}

A sequence $a_0, a_1, \ldots, a_n,\ldots$ of nonnegative real numbers is called a \emph{logarithmically concave} sequence, or a \emph{log-concave} sequence for short, if $a_i^2 > a_{i-1}a_{i+1}$ holds for $i > 1$.
Furthermore, it follows easily that every log-concave sequence is unimodal.
Menon in 1969 proves that the binomial convolution of two log-concave sequences is a log-concave sequence (see \cite{Me}).
Hence, we have the following consequence of Theorems \ref{t:k_nm} and \ref{t:S_rt}.

\begin{corollary}\label{c:Knm-Srt}
  The strong alliance polynomials of complete bipartite and double star graphs are unimodal (log-concave).
\end{corollary}

The Example \ref{ex:S44} shows that no all strong alliance polynomial are unimodal.

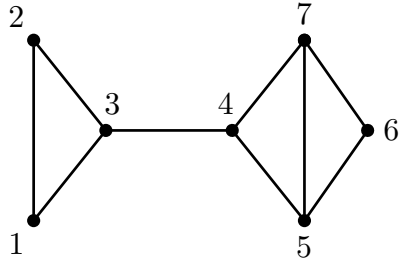
\begin{figure}[h]
\centering
\scalebox{1.2}
{\begin{pspicture}(-1.9,-1.5)(2.6,1.5)
\psline[linewidth=0.03cm,dotsize=0.07055555cm 2.5]{*-*}(-1.5,-1)(-1.5,1)
\psline[linewidth=0.03cm,dotsize=0.07055555cm 2.5]{-}(-1.5,-1)(-0.7,0)(-1.5,1)
\psline[linewidth=0.03cm,dotsize=0.07055555cm 2.5]{*-*}(-0.7,0)(0.7,0)(1.5,1)
\psline[linewidth=0.03cm,dotsize=0.07055555cm 2.5]{*-*}(0.7,0)(1.5,-1)(2.2,0)
\psline[linewidth=0.03cm,dotsize=0.07055555cm 2.5]{*-*}(1.5,-1)(1.5,1)
\psline[linewidth=0.03cm,dotsize=0.07055555cm 2.5]{-}(1.5,1)(2.2,0)
\uput[130](-1.5,1){$2$}
\uput[230](-1.5,-1){$1$}
\uput[75](-0.7,0){$3$}
\uput[105](0.7,0){$4$}
\uput[270](1.5,-1){$5$}
\uput[0](2.2,0){$6$}
\uput[90](1.5,1){$7$}
\end{pspicture}}
\caption{Graph which its alliance polynomial is not unimodal.} \label{fig:NoUnim}
\end{figure}

\begin{example}\label{ex:S44}
 Let us consider $G$ the graph in Figure \ref{fig:NoUnim}. It is easy to compute the coefficients of its strong alliance polynomial.
 \begin{description}
   \item[$a_{2}(G)=1$] {By Theorem \ref{th:prop} (iv).}
   \item[$a_{3}(G)=3$] {Since the strong defensive alliances with cardinality $3$ are $\{1,2,3\}$, $\{4,5,7\}$ and $\{5,6,7\}$.}
   \item[$a_{4}(G)=1$] {Since the strong defensive alliance with cardinality $4$ is $\{4,5,6,7\}$.}
   \item[$a_{5}(G)=4$] {Since the strong defensive alliances with cardinality $5$ are $\{1,3,4,5,6\}$, $\{1,3,4,5,7\}$, $\{2,3,4,5,6\}$ and $\{2,3,4,5,7\}$.}
   \item[$a_{6}(G)=5$] {Since the strong defensive alliances with cardinality $6$ are $\{2,3,4,5,6,7\}$, $\{1,3,4,5,6,7\}$, $\{1,2,3,4,6,7\}$, $\{1,2,3,4,5,7\}$ and $\{1,2,3,4,5,6\}$.}
   \item[$a_{7}(G)=1$] {By Theorem \ref{th:prop} (iii).}
 \end{description}
 Thus, we obtain $a(G;x) = x^2 + 3x^3 + x^4 + 4x^5 + 5x^6 + x^7$. Note that, its coefficients verify $a_{2}(G) < a_{3}(G)$, $a_{3}(G) > a_{4}(G)$ and $a_{4}(G) < a_{5}(G)$.
\end{example}

As usual, we define the graph $G\setminus e$ as the graph with $V(G\setminus e)=V(G)$ and $E(G\setminus e)=E(G) \setminus\{e\}$, where $e \in E(G)$.
Analogously, for $r\ge2$ we define $G\setminus \{e_1,\ldots,e_r\}$ as the graph with $V(G\setminus \{e_1,\ldots,e_r\})=V(G)$ and $E(G\setminus \{e_1,\ldots,e_r\})=E(G) \setminus\{e_1,\ldots,e_r\}$, where $\{e_1,\ldots,e_r\}\subset E(G)$.

\begin{theorem}\label{t:kn-e1e2}
Let $n\ge3$. Then, we have the following statements
\begin{enumerate}
\item  If $n$ is odd, then $ a(K_n \setminus e_1; x)\neq a(K_n; x)$.
\item  If $n$ is even, then $ a(K_n; x) = a(K_n\setminus e_1; x) = ... = \, a(G \setminus \{ e_1,\ldots,e_{\frac{n}{2}-1}\}; x)$, where edges $e_i$ are not common endpoints.
\end{enumerate}
\end{theorem}

Note that Theorem \ref{t:kn-e1e2} (2) provides different graphs which have same strong alliance polynomials.

\begin{proof}
${}$
\begin{enumerate}
\item Let $p$ be the natural number such that $n=2p+1$. By Proposition \ref{p:computos} (3), we have $a(K_n ) = p+1$ and every $S\subset V(K_n)$ with $|S|\ge p+1$ is strong defensive alliance in $K_n$; in particular, $a_{p+1}(K_n)={n\choose p+1}$.
    Denote $e_1:=[u,v]$.
    However, we choice $S\subset V(K\setminus e_1)$ with $|S|=p+1$ and $\{u,v\}\subset S$, but $S$ is no strong defensive alliance since $\d _S(v) = p-1 < p =$ $\d _{\overline{S}}(v)$. Thus, $a_{p+1}(K_n\setminus e_1)< {n\choose p+1}$ and so, $ a(K_n\setminus e_1; x)\neq a(K_n; x)$.

\item Let $p$ be the natural number with $n=2p$. Consider $G_r$ a graph obtained from $K_n$ deleting $1\le r\le p$ edges without common vertex.
    Our next claim is that any subset of $V(G_r)$ with cardinality greater than $p+1$ is a strong defensive alliance in $G_r$.
    Let us consider $S\subset V(G_r)$ with $|S|\ge p+1$. Then, we have $\d_S(v) \ge |S|-2 \ge p-1$ and $\d _{\overline{S}}(v) \le 2p-|S| \le p-1$ for every $v \in S$; in fact, $S$ is a strong defensive alliance in $G_r$.
    So, we have proved the claim and obtain $a(K_n;x)$ as addend in $a(G_r;x)$.

    We next claim is that if $r<p$ then $a(G_r)=p+1$.
    Seeking for a contradiction, assume that there is a strong defensive alliance $S\subset V(G_r)$ with $|S|\leq p$.
    Note that if $|S|<p$, then $\d_S(v)\le |S|-1 < p-1$ and $\d_{\overline{S}}(v) = \d(v)-\d_{S}(v) \ge 2p-2 - |S|+1 > p-1$ for every $v\in S$; in fact, $S$ is no a strong defensive alliance in $G_r$.
    Therefore, it suffices suppose that $|S|=p$.
    Hence, if $|S|=p$, then $\d_S(v)\le p-1$ and $\d_{\overline{S}}(v) = n-2 - \d_S(v) \ge p-1$ for every $v\in S$. Thus, $\d_S(v) = p-1=\d_{\overline{S}}(v)$ and $\d(v)=n-2$ for every $v\in S$. So, $r=p$. But, this is the contradiction we were looking for. Then, we obtain
    $$ a(K_n; x) = a(K_n\setminus\{e_1\};x) = ... = \, a(K_n \setminus \{e_1,\ldots,e_{p-1}\};x).$$
\end{enumerate}
\end{proof}

Note that if $\{ e_1,\ldots,e_{\frac{n}2}\}$ is a subset of $E(K_n)$ without common endpoints, then $a(K_n \setminus \{ e_1,\ldots,e_{\frac{n}2}\})=n/2$ since we may choose a strong defensive alliance $S$ in $K_n \setminus \{ e_1,\ldots,e_{\frac{n}2}\}$ with $|S|=p$ when $\{e_1,\ldots,e_{\frac{n}2}\}$ is a set of edges which have one endpoint in $S$ and the other endpoint in $\overline{S}$.
Besides, any $S\subset V(K_n \setminus \{ e_1,\ldots,e_{\frac{n}2}\})$ with cardinality at most $\frac{n}2-1$ is no strong defensive alliance.

\begin{remark}\label{r:Kn-ne}
If $n$ is even, then $ a(K_n; x) \neq a(K_n \setminus \{ e_1,\ldots,e_{\frac{n}2}\}; x)$, where $\{e_1,\ldots,e_{\frac{n}2}\} \subset E(K_n)$ without common endpoints.
For instance, $a(K_4;x)\neq a(C_4;x)$.
\end{remark}

\section*{Acknowledgements}
This work was partly supported by a grant for Mobility of own research program at the University Carlos III de Madrid and a grant from CONACYT (CONACYT-UAG I0110/62/10), M\'exico.

\

\end{document}